\documentclass[11pt,leqno]{amsart}
\usepackage{a4wide}
\usepackage{amssymb,upgreek}
\usepackage{amsmath}
\usepackage{amsthm}
\usepackage{mathrsfs}
\usepackage{bm,euscript,csquotes,comment}
\usepackage[curve]{xypic}
\usepackage[usenames,dvipsnames]{xcolor}

\usepackage{hyperref}

\theoremstyle{plain}

\newcommand{\id}{\operatorname{id}}

\newcommand{\op}{\operatorname{op}}
\newcommand{\pr}{\operatorname{pr}}

\newcommand{\res}{\operatorname{res}}

\newcommand{\Hom}{\operatorname{Hom}}

\newcommand{\Ext}{\operatorname{Ext}}

\newcommand{\Mod}{\operatorname{Mod}}

\newcommand{\ind}{\operatorname{ind}}

\newcommand{\chara}{\operatorname{char}}

\newtheorem{theorem}{Theorem}[section]
\newtheorem{corollary}[theorem]{Corollary}
\newtheorem{lemma}[theorem]{Lemma}
\newtheorem{proposition}[theorem]{Proposition}

\theoremstyle{definition}
\newtheorem{remark}[theorem]{Remark}

\title{\textbf{Local coherence for representations of amalgams}}
\author{Peter Schneider}
\date{\today}

\address{ Universit\"at M\"unster,  Mathematisches Institut,  Einsteinstr. 62, 48291 M\"unster, Germany}
\email{pschnei@uni-muenster.de}
\urladdr{http://www.uni-muenster.de/math/u/schneider/}

\begin{document}

\maketitle

\tableofcontents

\section{Introduction}\label{sec:intro}

Throughout this paper $G$ is a locally profinite group. For most of the first four sections we assume that
\begin{multline*}
  \text{(A) \ $G$ is the amalgamation of two compact open subgroups $K_1, K_2 \subseteq G$} \\
   \text{over their intersection $I := K_1 \cap K_2$.}
\end{multline*}
This means that $G$ is the colimit in the category of discrete groups of the diagram
\begin{equation*}
  \xymatrix@R=0.5cm{
                &         K_1     \\
  I \ar[ur]^{\subseteq} \ar[dr]_{\subseteq}                 \\
                &         K_2                }
\end{equation*}
(cf.\ \cite{Ser} Chap.\ I \S 1).\\

\textit{A basic example:} Let $F/\mathbb{Q}_p$ be a finite extension with ring of integers $\mathcal{O}$ and prime element $\varpi$. The conjugacy classes of maximal compact subgroups in $G := SL_(F)$ are represented by $K_1 := SL_2(\mathcal{O})$ and 
$K_2 := (\begin{smallmatrix}
\varpi^{-1} & 0 \\
0     & 1
\end{smallmatrix})
K_1 
(\begin{smallmatrix}
\varpi & 0 \\
0     & 1
\end{smallmatrix})$. Their intersection $I = K_1 \cap K_2$ is an Iwahori subgroup. By Ihara's theorem $G$ is the amalgamation of $K_1$ and $K_2$ (cf.\ \cite{Ser} Cor.\ II.1.1).\\

We also fix a field $k$ and let $\Mod(G)$ denote the Grothendieck abelian category of smooth $G$-representations in $k$-vector spaces. We recall that a $G$-representation $V$ is called smooth if, for any $v \in V$, its stabilizer $\{g \in G : gv = v\}$ is open in $G$.

\textbf{Assuming} (A) for the rest of this section we observe that restricting the group action gives rise to a commutative diagram of functors:
\begin{equation}\label{diag:lim}
\xymatrix@R=0.5cm{
                            & \Mod(K_1) \ar[dr]^{res}             \\ 
\Mod(G) \ar[ur]^{res} \ar[dr]_{res} &    & \Mod(I)          \\
                            & \Mod(K_2) \ar[ur]_{res}               }
\end{equation}
To see that this actually is a limit diagram we briefly recall Gabriel's construction of the limit as a "recollement'' $\mathcal{R}$ (\cite{Gab} p. 440). The objects of the category $\mathcal{R}$ are triples $(V_1,V_2, \sigma)$ where $V_i$ is an object in $\Mod(K_i)$ and $\sigma : V_1 \xrightarrow{\cong} V_2$ is an isomorphism on $\Mod(I)$; the morphisms are the obvious ones. Projections give rise to obvious functors $\pr_i \mathcal{R} \longrightarrow \Mod(K_i)$ such that the diagram
\begin{equation*}
\xymatrix@R=0.5cm{
                            & \Mod(K_1) \ar[dr]^{res}             \\ 
\mathcal{R} \ar[ur]^{\pr_1} \ar[dr]_{pr_2} &    & \Mod(I)          \\
                            & \Mod(K_2) \ar[ur]_{res}               }
\end{equation*}
is commutative up to natural isomorphism. This diagram has an obvious universal property saying that it is a limit diagram. We also have the exact functor $R : \Mod(G) \longrightarrow \mathcal{R}$ sending $V$ to $(V_{|K_1}, V_{|K_2}, \id)$ such that $\pr_i \circ R$ is the restriction to $K_i$. Obviously $R$ is faithful. To see that $R$ actually is an equivalence of categories let $(V_1,V_2,\sigma)$ be triple in $\mathcal{R}$. We equip $V := V_1$ with a $K_2$-action via
\begin{equation*}
  g v := \sigma^{-1}(g(\sigma v))) \qquad\text{$v \in V_1$ and $g \in K_2$}.
\end{equation*}
In this way $V$ carries a smooth $K_1$-action as well as a smooth $K_2$-action which coincide on $I$. By the colimit property of $G$ this corresponds to a unique extension of the these $K_i$-actions to a smooth $G$-action on $V$. Moreover
\begin{equation*}
  (\id,\sigma) : (V,V,\id) \xrightarrow{\cong} (V_1,V_2, \sigma)
\end{equation*}
is an isomorphism in $\mathcal{R}$. This shows that \eqref{diag:lim} is a limit diagram. This immediately implies that, for any $V_i$ in $\Mod(K_i)$, the sequence
\begin{equation}\label{f:basic-sequence}
  \xymatrix{
    0 \ar[r] & \Hom_G(V_1,V_2) \ar[rr]^-{(\res^G_{K_1},\res^G_{K_2})} && \Hom_{K_1}(V_1,V_2) \times \Hom_{K_2}(V_1,V_2) \ar[rr]^-{\res^{K_1}_I - \res^{K_2}_I} && \Hom_I(V_1,V_2)  }
\end{equation}
is exact. It is the starting point for the proof of the first main result. In order to simplify notation we write $\Ext^j_G$ for the Ext-groups in the category $\Mod(G)$.

\begin{theorem}\label{thm:MV}
  For any $V_1, V_2$ in $\Mod(G)$ we have the functorial long exact sequence
\begin{align*}
  0 \rightarrow & \Hom_G(V_1,V_2) \rightarrow \Hom_{K_1}(V_1,V_2) \times \Hom_{K_2}(V_1,V_2) \rightarrow \Hom_I(V_1,V_2) \rightarrow  \\
        & \Ext^1_G(V_1,V_2) \rightarrow  ...    \\
        & \ldots \\
        & \ldots  \\
        & \Ext^j_G(V_1,V_2) \rightarrow \Ext^j_{K_1}(V_1,V_2) \times \Ext^j_{K_2}(V_1,V_2) \rightarrow \Ext^j_I(V_1,V_2) \rightarrow \qquad\qquad\qquad \\
        & \Ext^{j+1}_G(V_1,V_2) \rightarrow  ...  
\end{align*}
\end{theorem}
The proof will be given in section \ref{sec:MV}. The main technical tool for it are the functors of compact induction whose for our purposes specific properties will be developed in section \ref{sec:compact-ind}. Our second main result, whose proof is contained in section \ref{sec:loc-coh}, then is the following.

\begin{theorem}\label{thm:loc-coh}
   The category $\Mod(G)$ is locally coherent.
\end{theorem}

In \cite{Sho} this theorem has been proved under the additional assumptions that $G$ is a locally analytic $p$-adic group and $k$ has characteristic $p$. The proof is completely different from the one in this paper. It crucially uses that certain completed group rings under these additional assumptions are noetherian (by a theorem of Lazard). Therefore the proof in \cite{Sho} cannot be generalized to the setting of this paper.

In the last section \ref{sec:vista} we discuss generalizations of our results to groups $G = \mathbf{G}(L)$ of $L$-points, where $L/\mathbb{Q}_p$ is a finite extension and  $\mathbf{G}$ is a semisimple simply connected reductive group over $L$. For example, in this case the long exact $\Ext$-sequence in Theorem \ref{thm:MV} has to be replaced by a spectral sequence.

The main interest in the question for which groups $G$ the category $\Mod(G)$ is locally coherent comes from the $p$-adic local Langlands program - see \cite{EGH} Conjecture 6.1.15.

This research was funded by the Deutsche Forschungsgemeinschaft (DFG, German Research Foundation) – Project-ID 427320536 – SFB 1442, as well as under Germany's Excellence Strategy EXC 2044/2 –390685587, Mathematics Münster: Dynamics–Geometry–Struc- ture.

\section{Compact induction}\label{sec:compact-ind}

Let $K \subseteq G$ be an arbitrary compact open subgroup. The compact induction functor
\begin{equation*}
  \ind_K^G : \Mod(K) \longrightarrow \Mod(G)
\end{equation*}
is given by
\begin{align*}
  \ind_K^G(V) := &\ \text{all locally constant compactly supported functions $f : G \rightarrow V$}  \\
                 & \ \text{satisfying}\ f(hg) = hf(g)\ \text{for any $h \in K$ and $g \in G$}.
\end{align*}
An element $g \in G$ acts on $f \in \ind_K^G(V)$ by ${^g}f(g') = f(g' g)$. The following properties are easy to verify (cf.\ \cite{Vig} I.5):
\begin{itemize}
  \item[a)] The functor $\ind_K^G$ is exact.
  \item[b)] The functor $\ind_K^G$ is left adjoint to the restriction functor $\res_K^G$.
  \item[c)] For any $V$ in $\Mod(G)$ one has the natural $G$-equivariant isomorphism
\begin{align*}
  \ind_K^G(k) \otimes_k V & \xrightarrow{\;\cong\;} \ind_K^G(\res_K^G(V)) \\
              f \otimes v & \longmapsto F_{f \otimes v}(g) := f(g) gv 
\end{align*}
with inverse $F \longmapsto \sum_{h \in K \backslash G} \chara_{Kh} \otimes\, h^{-1} F(h)$, where $G$ acts diagonally on the left hand side. Here $\chara_A \in \ind_K^G(k)$ denotes, for any left $K$-invariant subset $A \subseteq G$, the characteristic function of $A$.
\end{itemize}

For the rest of this section we \textbf{assume} (A) and let $V_1$ and $V_2$ be two objects in $\Mod(G)$. By using property b) we may rewrite the exact sequence \eqref{f:basic-sequence} in the following form:
\begin{multline}\label{f:Hom-ind-sequence}
  0 \rightarrow \Hom_G(V_1,V_2) \rightarrow \Hom_G(\ind_{K_1}^G(V_1),V_2) \times \Hom_G(\ind_{K_1}^G(V_1),V_2) \\
      = \Hom_G(\ind_{K_1}^G(V_1) \oplus \ind_{K_2}^G(V_1),V_2)  \rightarrow \Hom_G(\ind_I^G(V_1),V_2)
\end{multline}
Here and in the following we will usually omit the restriction functor to a compact open subgroup from the notation.
The task at hand is to identify the maps in this sequence. For this we have to recall first the Frobenius isomorphism
\begin{equation*}
  \mathcal{R}_{U,V_2} : \Hom_K(U,V_2) \xrightarrow{\cong} \Hom_G(\ind_K^G(U),V_2)
\end{equation*}
for any $U$ in $\Mod(K)$. We have the map
\begin{align*}
  \iota : U & \longrightarrow \ind_K^G(U) \\
           u & \longmapsto i_u(g) :=
           \begin{cases}
           gu & \text{if $g \in K$}, \\
           0 & \text{otherwise}.
           \end{cases}
\end{align*}

\begin{remark}
$\iota$ is $K$-equivariant.
\end{remark}
\begin{proof}
For $h \in K$ we compute
\begin{equation*}
  \iota(hu)(g) = i_{hu}(g) = \begin{cases}
           ghu & \text{if $g \in K$}, \\
           0 & \text{otherwise}
            \end{cases}
\end{equation*}
and
\begin{equation*}
  ({^h \iota}(u))(g) = ({^h i_u})(g) = i_u(gh) =
  \begin{cases}
           ghu & \text{if $g \in K$}, \\
           0 & \text{otherwise}.
            \end{cases}
\end{equation*}
\end{proof}
We see that any $f \in \ind^G_K(U)$ can be written as the sum $f = \sum_{g \in K \backslash G} {^{g^{-1}}} i_{f(g)}$.
The reciprocity isomorphism $\mathcal{R}_{U,V_2}$ is given by (compare \cite{Vig} formula (3) on p.\ 39)
\begin{equation*}
  \mathcal{R}_{U,V_2}(S)(f) = \sum_{g \in K \backslash G} \mathcal{R}_{U,V_2}(S)( {^{g^{-1}}} i_{f(g)}) = \sum_{g \in K \backslash G} g^{-1}\mathcal{R}_{U,V_2}(S)(\iota(f(g))) = \sum_{g \in K \backslash G} g^{-1}S(f(g)) .
\end{equation*}
For $V_1$ instead of $U$ and $S \in \Hom_G(V_1,V_2)$ this simplifies to
\begin{equation}\label{f:R}
   \mathcal{R}_{V_1,V_2}(S)(f) = S(\sum_{g \in K \backslash G} g^{-1}f(g)) .
\end{equation}
Next we introduce the map
\begin{align*}
  \pi := \pi_{G,K} : \ind_K^G(V_1) & \longrightarrow V_1 \\
                    f & \longmapsto \sum_{g \in K \backslash G} g^{-1} f(g)  \ .
\end{align*}
This is well defined since the defining sum is independent of the choice of the coset representatives.

\begin{remark}
   $\pi_{G,K}$ is $G$-equivariant.
\end{remark}
\begin{proof}
For $g_0 \in G$ we compute
\begin{multline*}
  \pi({^{g_0}f}) = \sum_{g \in K \backslash G} g^{-1} ({^{g_0}f})(g) = \sum_{g \in K \backslash G} g^{-1} f(g g_0) = \sum_{g \in K \backslash G} g_0 g^{-1} f(g) \\
    = g_0 \sum_{g \in K \backslash G} g^{-1} f(g) = g_0 \pi(f) \ .
\end{multline*}
\end{proof}

The formula \eqref{f:R} now can be rewritten as
\begin{equation*}
  \mathcal{R}_{V_1,V_2}(S)(f) = S(\pi_{G,K}(f)) \ .
\end{equation*}

This means that the diagram
\begin{equation}\label{diag:R}
  \xymatrix{
    &&& \Hom_K(V_1,V_2) \ar[dd]_{\cong}^{\mathcal{R}_{V_1,V_2}} \\
    \Hom_G(V_1,V_2) \ar[urrr]^{\res^G_K} \ar[drrr]^-{\Hom_G(\pi_{G,K},V_2)}  \\
      &&& \Hom_G(\ind_K^G(V_1),V_2)   }
\end{equation}
is commutative. The upper oblique arrow is obviously injective. Hence the lower one is injective, too. This also follows from the

\begin{remark}
   $\pi_{G,K} \circ \iota = \id$ .
\end{remark}
\begin{proof}
For $v \in V_1$ we compute $\pi_{G,K} (\iota(v)) = \pi (i_v) = \sum_{g \in K \backslash G} g^{-1} i_v(g) = i_v(1) = v$.
\end{proof}

Now we have everything in place to determine the maps in the sequence \eqref{f:Hom-ind-sequence}. We finally introduce, for $i = 1, 2$. the $G$-equivariant maps
\begin{equation*}
  \gamma_i : \ind_I^G(V_1) = \ind_{K_i}^G(\ind_I^{K_i}(V_1)) \xrightarrow{\ind_{K_i}^G(\pi_{K_i,I})} \ind_{K_i}^G(V_1) \ .
\end{equation*}

\begin{proposition}\label{prop:dual-sequence}
  The exact sequence \eqref{f:Hom-ind-sequence} arises, by applying the functor $\Hom_G(- , V_2)$, from the exact sequence
\begin{equation*}
  \ind_I^G(V_1) \xrightarrow{(\gamma_1,-\gamma_2)} \ind_{K_1}^G(V_1) \oplus \ind_{K_2}^G(V_1) \xrightarrow{\pi_{G,K_1} + \pi_{G,K_2}} V_1 \longrightarrow 0 
\end{equation*}
with $\gamma_i$ explicitly given by $\gamma_i(f)(g) = \sum_{h \in I \backslash K_i} h^{-1} f(hg)$.
\end{proposition}
\begin{proof}
We have the commutative diagrams
\begin{equation*}
  \xymatrix{
      &&& \Hom_I(V_1,V_2) \ar[d]_{\cong}^{\mathcal{R}_{V_1,V_2}} \\
    \Hom_{K_i}(V_1,V_2) \ar[d]^{\cong}_{\mathcal{R}_{V_1,V_2}} \ar[urrr]^{\res_I^{K_i}}  \ar[rrr]_-{\Hom_{K_i}(\pi_{K_i,I},V_2)} &&& \Hom_{K_i}(\ind_I^{K_i}(V_1),V_2) \ar[d]_{\cong}^{\mathcal{R}_{\ind_I^{K_i}(V_1),V_2)}} \\
    \Hom_G(\ind_{K_i}^G(V_1),V_2)  \ar[rrr]_-{\Hom_G(\ind_{K_i}^G(\pi_{K_i,I}),V_2)} &&& \Hom_G(\ind_{K_i}^G(\ind_I^{K_i}(V_1),V_2) \ar[d]_{\cong}^{\Hom_G(\widetilde{.},V_2)} \\
     &&& \Hom_G(\ind_I^G(V_1),V_2) ,  }
\end{equation*}
where the bottom perpendicular arrow uses the transitivity isomorphism (\cite{Vig} I.5.3)
\begin{align*}
  \ind_I^G(V_1) & \xrightarrow{\;\cong\;} \ind_{K_i}^G(\ind_I^{K_i}(V_1)) \\
              f & \longmapsto \widetilde{f}(g) := {^g f} | K_i
\end{align*}
whose inverse is given by $F \longmapsto [g \mapsto F(g)(1)]$. The upper triangle is the diagram \eqref{diag:R} for the groups $I \subseteq K_i$. The middle square is commutative by the naturality of the Frobenius isomorphism. For the left arrow it remains to compute
\begin{multline*}
  \gamma_i(f)(g) = \ind_{K_i}^G(\pi_{K_i,I})(\widetilde{f})(g) = [g \mapsto \pi_{K_i,I} ({^g f}|K_i)]  \\
   = [g \mapsto   \sum_{h \in I \backslash K_i} h^{-1} ({^g f})(h)] = \sum_{h \in I \backslash K_i} h^{-1} f(hg) \ .
\end{multline*}
This proves the assertion for the left arrow. The assertion for the right arrow is immediate from the diagram \eqref{diag:R}.

The exactness follows from universality.
\end{proof}

\section{Mayer-Vietoris sequence}\label{sec:MV}

Throughout this section we \textbf{assume} (A). The technical task is to show that the sequence in Prop.\ \ref{prop:dual-sequence} actually is a short exact sequence, i.e., that the map $(\gamma_1, \gamma_2)$ is injective. This requires the fact that with the amalgamated group $G$ there is associated a tree on which $G$ acts. We recall this from \cite{Ser} Thm.\ I.4.1.6.

\begin{theorem}\label{thm:tree}
   The amalgamated group $G$ acts, without inversion and transitively on the edges, on a tree $\mathcal{T}$ such that for some edge $e$ with vertices $v_1$ and $v_2$ the subgroups $K_i$ and $I$ are the stabilizers of the vertices $v_i$ and the edge $e$, respectively.
\end{theorem}

Let $edges(\mathcal{T})$ and $vertices(\mathcal{T})$ denote the set of edges and the set of vertices of $\mathcal{T}$, respectively. We then have $G$-equivariant bijections
\begin{align}\label{f:edge}
  I \backslash G & \xrightarrow{\;\sim\;} edges(\mathcal{T})  \\
  Ig & \longmapsto g^{-1} e       \nonumber
\end{align}
and
\begin{align}\label{f:vertex}
   K_1 \backslash G \ \dot\cup\  K_2 \backslash G & \xrightarrow{\;\sim\;} vertices(\mathcal{T}) \\
                                            K_i g & \longmapsto g^{-1} v_i \ .                   \nonumber
\end{align}
We use this to establish the special case where $V_1 = k$ is the trivial $G$-representation first.

\begin{lemma}\label{lem:gamma-k}
   For any $i = 1, 2$ and any $g \in G$ we have $\gamma_i(\chara_{Ig}) = \chara_{K_i g}$.
\end{lemma}
\begin{proof}
Using the description of $\gamma_i$ in Prop.\ \ref{prop:dual-sequence} we compute
\begin{align*}
  \gamma_i(\chara_{Ig})(g_0) & = \sum_{h \in I \backslash K_i} h^{-1} \chara_{Ig}(hg_0) =  \sum_{h \in I \backslash K_i, hg_0 \in Ig} 1  
   = \begin{cases}
           1 & \text{if $gg_0^{-1} \in K_i$}, \\
           0 & \text{otherwise}
            \end{cases}   \\
           & = \chara_{K_i g}(g_0) \ .
\end{align*}
\end{proof}

\begin{proposition}\label{prop:MV-trivial}
   The map $\ind_I^G(k) \xrightarrow{(\gamma_1,\gamma_2)} \ind_{K_1}^G(k) \oplus \ind_{K_2}^G(k)$ is injective.
\end{proposition}
\begin{proof}
Obviously $B := \{\chara_{Ig} : g \in I \backslash G\}$ is a $k$-basis of $\ind_I^G(k)$. Lemma \ref{lem:gamma-k} implies that $\gamma_i(B)$ is a $k$-basis of $\ind_{K_i}^G(k)$. The asserted injectivity therefore is equivalent to the claim that, for any finite subset $C \subseteq B$, the set $\{\gamma_1(b) + \gamma_2(b) : b \in C\}$ is linearly independent in $\ind_{K_1}^G(k) \oplus \ind_{K_2}^G(k)$. Suppose therefore that
\begin{equation*}
    0 = \sum_{b \in C} r_b(\gamma_1(b) + \gamma_2(b)) = \sum_{b \in C} r_b \gamma_1(b) + \sum_{b \in C} r_b \gamma_2(b)
\end{equation*}
for some $r_b \in k$. It follows that
\begin{equation}\label{f:equation}
  \sum_{b \in C} r_b \gamma_i(b) = 0  \qquad\text{for $i = 1, 2$}.
\end{equation}

It is now convenient to make the identifications \eqref{f:edge} and \eqref{f:vertex}. By Lemma \ref{lem:gamma-k} the maps $\gamma_1$ and $\gamma_2$ become the maps which send an edge to its two extremities. We fix a vertex \textbf{o} of $\mathcal{T}$. For any edge $b \in C$ there is a unique geodesic in $\mathcal{T}$ starting at \textbf{o} and having $b$ as its final edge. Let $v(b)$ denote the vertex of $b$ in which this geodesic ends. Now we start with a $b \in C$ such that the corresponding geodesic has maximal length. Then there cannot exist any edge in $B$ different from $b$ which has $v_b$ as a vertex. Therefore \eqref{f:equation} implies that $r_b = 0$. So we may remove $b$ from $C$ and repeat this argument. In this way it follows inductively that $r_b = 0$ for all $b \in C$.
\end{proof}

To deduce the general case we use the isomorphism in property c) above.

\begin{lemma}\label{lem:property-c}
   For any $V_1$ in $\Mod(G)$ and any $i = 1, 2$ the diagram
\begin{equation*}
  \xymatrix{
  \ind_I^G(V_1) \ar[d]_{\gamma_i}  & \ind_I^G(k) \otimes_k V_1 \ar[l]_-{\cong} \ar[d]^{\gamma_i \otimes \id} \\
  \ind_{K_i}(V_1) \ar[r]^-{\cong} & \ind_{K_i}^G(k) \otimes_k V_1 ,   } 
\end{equation*}
where the horizontal isomorphisms are the ones from property c), is commutative.
\end{lemma}
\begin{proof}
Let $f \otimes v \in \ind_I^G(k) \otimes_k V_1$. Under the upper horizontal isomorphism it is mapped to $F_{f \otimes v}$. Recalling the formula for $\gamma_i$ in Prop.\ \ref{prop:dual-sequence} we compute
\begin{equation*}
  F_i(g) := \gamma_i(F_{f \otimes v})(g) = \sum_{\kappa \in I \backslash K_i} \kappa^{-1} F_{f \otimes v}(\kappa g) =
   \sum_{\kappa \in I \backslash K_i} \kappa^{-1} f(\kappa g) \kappa gv = \sum_{\kappa \in I \backslash K_i} f(\kappa g) gv   \ .
\end{equation*}
The lower horizontal isomorphism sends $F_i$ to $\sum_{h \in K_i \backslash G} \chara_{K_i h} \otimes h^{-1} F_i(h)$. We finally compute the latter to be equal to
\begin{align*}
   & \sum_{h \in K_i \backslash G} (\chara_{K_i h} \otimes \sum_{\kappa \in I \backslash K_i} h^{-1} f(\kappa h) hv) \\
   & = \sum_{h \in K_i \backslash G} (\chara_{K_i h} \otimes  \sum_{\kappa \in I \backslash K_i}f(\kappa h) v) \\
   & = (\sum_{h \in K_i \backslash G} \chara_{K_i h} \cdot (\sum_{\kappa \in I \backslash K_i}f(\kappa h))) \otimes   v \\
   & = (\sum_{h \in K_i \backslash G} \chara_{K_i h} \cdot \gamma_i(f)(h)) \otimes v  \\
   & = \gamma_i(f) \otimes v \ .
\end{align*}
\end{proof}

\begin{proposition}\label{prop:MV-general}
  For any $V_1$ in $\Mod(G)$ we have the short exact sequence
\begin{equation*}
  0 \longrightarrow \ind_I^G(V_1) \xrightarrow{(\gamma_1,-\gamma_2)} \ind_{K_1}^G(V_1) \oplus \ind_{K_2}^G(V_1) \xrightarrow{\pi_{G,K_1} + \pi_{G,K_2}} V_1 \longrightarrow 0 \ .
\end{equation*}
\end{proposition}
\begin{proof}
Combine Prop.\ \ref{prop:dual-sequence}, Prop.\ \ref{prop:MV-trivial}, and Lemma \ref{lem:property-c}.
\end{proof}

\begin{theorem}\label{thm:MV2}
  For any $V_1, V_2$ in $\Mod(G)$ we have the bi-functorial long exact sequence
\begin{align*}
  0 \rightarrow & \Hom_G(V_1,V_2) \rightarrow \Hom_{K_1}(V_1,V_2) \times \Hom_{K_2}(V_1,V_2) \rightarrow \Hom_I(V_1,V_2) \rightarrow  \\
        & \Ext^1_G(V_1,V_2) \rightarrow  ...    \\
        & \ldots \\
        & \ldots  \\
        & \Ext^j_G(V_1,V_2) \rightarrow \Ext^j_{K_1}(V_1,V_2) \times \Ext^j_{K_2}(V_1,V_2) \rightarrow \Ext^j_I(V_1,V_2) \rightarrow \qquad\qquad\qquad \\
        & \Ext^{j+1}_G(V_1,V_2) \rightarrow  ...  
\end{align*}
\end{theorem}
\begin{proof}
Applying the Hom-functor $\Hom_G(-,V_2)$ to the short exact sequence in Prop.\ \ref{prop:MV-general} results, by Prop.\ \ref{prop:dual-sequence}, in the long exact Ext-sequence
\begin{align*}
  0 \rightarrow & \Hom_G(V_1,V_2) \rightarrow \Hom_G(\ind_{K_1}^G(V_1),V_2) \times \Hom_G(\ind_{K_1}^G(V_1),V_2) \rightarrow \Hom_G(\ind_I^G(V_1),V_2)  \\
        & \ldots  \\
     \rightarrow    & \Ext^j_G(V_1,V_2) \rightarrow \Ext^j_G(\ind_{K_1}^G(V_1),V_2) \times \Ext^j_G(\ind_{K_1}^G(V_1),V_2) \rightarrow \Ext^j_G(\ind_I^G(V_1),V_2) \rightarrow  ...  
\end{align*}
But the restriction functors having an exact left adjoint preserve injective objects. Hence Frobenius reciprocity holds for the Ext-groups as well, and we obtain the asserted long exact sequence.
\end{proof}

\section{Local coherence}\label{sec:loc-coh}

The category $\Mod(G)$ is locally of finite type, since any object obviously is the direct limit of its finitely generated subobjects. We recall that an object $V$ in $\Mod(G)$ is called finitely presented if the functor $\Hom_G(V,-)$ commutes with direct limits.

\begin{lemma}\label{lem:generators}
   Let $\mathfrak{U}$ be a fundamental system of compact open subgroups of $G$. Then the representations $\ind_U^G(k)$, for $U \in \mathfrak{U}$, form a set of generators of $\Mod(G)$.
\end{lemma}
\begin{proof}
Let $\alpha : V \rightarrow V'$ be a nonzero map and choose a $v \in V$ such that $\alpha(v) \neq 0$ as well as a $U \in \mathfrak{U}$ which fixes $v$. Then $\alpha^U : V^U \rightarrow V'^U$ is nonzero. It remains to note that $\Hom_{\Mod(G)}(\ind_U^G(k),-) = (-)^U$.
\end{proof}

\begin{lemma}\label{lem:ind-fp}
   Let $U \subseteq G$ be a compact open subgroup and $M$ be a finite dimensional representation in $\Mod(U)$. Then $\ind_U^G(M)$ is finitely presented in $\Mod(G)$.
\end{lemma}
\begin{proof}
We have
\begin{equation*}
 \Hom_{\Mod(G)}(\ind_U^G(M),V) = \Hom_{\Mod(U)}(M,V) = (M^* \otimes_k V)^U \ .
\end{equation*}
But the cohomology of profinite groups commutes with direct limits.
\end{proof}

These two lemmas, in particular, show that $\Mod(G)$ is locally finitely presented. Let $\mathcal{FP}(G)$ denote the class of finitely presented objects in $\Mod(G)$. We further let $\mathcal{FP}$-inj$(G)$ denote the class of object $A$ in $\Mod(G)$ for which $\Ext_G^1(F,A) = 0$ for any $F$ in $\mathcal{FP}(G)$.

\begin{lemma}\label{lem:profinite}
   For a profinite group $H$ then we have:
\begin{itemize}
  \item[i.] An object in $\Mod(H)$ is finitely presented if and only if it is finite dimensional;
  \item[ii.] $\mathcal{FP}$-inj$(H)$ is the class of all injective objects in $\Mod(H)$.
\end{itemize}
\end{lemma}
\begin{proof}
i. This is an easy exercise. ii. See \cite{BGP} Remark 5.6.2.
\end{proof}

\begin{lemma}\label{lem:ind-res}
   For a compact open subgroup $K \subseteq G$ we have:
\begin{itemize}
  \item[i.] The compact induction functor $\ind_K^G$ respects the classes $\mathcal{FP}$;
  \item[ii.] The restriction functor $\res_K^G$ respects the classes $\mathcal{FP}$-inj.
\end{itemize}
\end{lemma}
\begin{proof}
i. This is immediate from Lemma \ref{lem:profinite}.i and Lemma \ref{lem:ind-fp}. ii. This follows from i. and the adjunction
\begin{equation*}
  \Ext_G^1(\ind_K^G(M),V) \cong \Ext_K^1(M,\res_K^G(V)) \ .
\end{equation*}
\end{proof}

For the remaining results we \textbf{assume} (A) again.

\begin{proposition}\label{prop:inj-dim}
   For any $V_1$ in $\Mod(G)$ and $V_2$ in $\mathcal{FP}$-inj$(G)$ we have
\begin{equation*}
  \Ext_G^j(V_1,V_2) = 0  \qquad\text{for any $j \geq 2$}
\end{equation*}
(i.e., the objects in $\mathcal{FP}$-inj$(G)$ have injective dimension $\leq 1$).
\end{proposition}
\begin{proof}
Theorem \ref{thm:MV2} gives us the exact sequence
\begin{equation*}
  \Ext_I^{j-1}(V_1,V_2) \longrightarrow \Ext_G^j(V_1,V_2) \longrightarrow \Ext_{K_1}^j(V_1,V_2) \oplus \Ext_{K_2}^j(V_1,V_2) \ .
\end{equation*}
By Lemmas \ref{lem:ind-res}.ii and \ref{lem:profinite}.ii the outer terms are zero for $j \geq 2$.
\end{proof}

\begin{corollary}\label{cor:coker}
   The class $\mathcal{FP}$-inj$(G)$ is closed under taking cokernels of monomorphisms.
\end{corollary}
\begin{proof}
Let $0 \rightarrow A \rightarrow B \rightarrow C \rightarrow 0$ be an exact sequence in $\Mod(G)$ with $A$ and $B$ being contained in $\mathcal{FP}$-inj$(G)$. For any $F$ in $\mathcal{FP}(G)$ we have the exact sequence
\begin{equation*}
    \Ext_G^1(F,B) \longrightarrow \Ext_G^1(F,C) \longrightarrow \Ext_G^2(F,A) \ .
\end{equation*}
The first term is zero by assumption and the last one by Prop.\ \ref{prop:inj-dim}. Hence $\Ext_G^1(F,C) = 0$.
\end{proof}

We recall that a locally finitely presented Grothendieck category is called locally coherent if the full subcategory of finitely presented objects is an abelian subcategory.

\begin{theorem}\label{thm:loc-coh2}
   The category $\Mod(G)$ is locally coherent.
\end{theorem}
\begin{proof}
Apply \cite{BGP} Thm.\ 5.5 based on Cor.\ \ref{cor:coker}.
\end{proof}

\section{Generalizations}\label{sec:vista}

In this section we consider the general case that $G = \mathbf{G}(L)$ is the group of $L$-points, where $L/\mathbb{Q}_p$ is a finite extension, of a semisimple simply connected reductive group $\mathbf{G}$ over $L$. The group $G$ acts on its Bruhat-Tits building $\mathcal{X}$, which has (among others) the following properties (\cite{Tit} \S 3):
\begin{itemize}
  \item[--] $\mathcal{X}$ is a $G$-CW-complex in the sense of \cite{Bro}.
  \item[--] $\mathcal{X}$ is contractible and hence simply-connected.
  \item[--] Any chosen chamber $\Delta$ of $\mathcal{X}$ is a fundamental domain for the $G$-action on $\mathcal{X}$. In particular, $G$ acts without inversion on the $1$-dimensional facets of $\mathcal{X}$.
\end{itemize}
Let $d$ denote the dimension of $\mathcal{X}$. For any facet $F$ of $\mathcal{X}$ we let $K_F \subseteq G$ denote the stabilizer of $F$. We fix a chamber $\Delta$. Now \cite{Bro} Thm.\ 3 tells us that $G$ is the colimit of the stabilizer groups $K_e$ and $K_v$, where $e$ and $v$ run over the edges and vertices of $\Delta$, respectively, with the obvious inclusion maps between them. But, indeed, we will have to take into account the group $K_F$ for all facets of $\Delta$ by considering the standard homology chain complex
\begin{equation}\label{f:chains}
  0 \xrightarrow{\quad} \bigoplus_{\dim(F) = d} k \xrightarrow{\ \partial\ } \bigoplus_{\dim(F) = d-1} k \xrightarrow{\ \partial\ } \ldots \xrightarrow{\ \partial\ } \bigoplus_{\dim(F) = 0} k \xrightarrow{\quad} k \xrightarrow{\quad} 0
\end{equation}
of the building $\mathcal{X}$. Since $G$ acts on $\mathcal{X}$ we find a $G$-invariant orientation of $\mathcal{X}$ and use it to define the boundary maps $\partial$ (cf.\ \cite{SS} \S II.1 for details). This makes \eqref{f:chains} into a complex of smooth $G$-representations (with the trivial $G$-action on $K$). Since $\mathcal{X}$ is contractible this is, in fact, an exact resolution of the trivial representation $k$ in $\Mod(G)$. The fact that $\Delta$ is a fundamental domain for the $G$-action on $\mathcal{X}$ allows to rewrite \eqref{f:chains} as a complex of compact inductions:
\begin{equation*}
  0 \rightarrow \ind_{K_\Delta}^G(k) \rightarrow \bigoplus_{F \subseteq \Delta, \dim(F) = d-1} \ind_{K_F}^G(k) \rightarrow \ldots \rightarrow \bigoplus_{F \subseteq \Delta,\dim(F) = 0} \ind_{K_F}^G(k) \rightarrow k \rightarrow 0
\end{equation*}
Given any object $V_1$ in $\Mod(G)$ we may tensor the above complex over $k$ by $V_1$ with the diagonal $G$-action. Using property c) from the beginning of section \eqref{sec:compact-ind} we finally obtain an exact resolution of $V_1$ in $\Mod(G)$ of the form
\begin{equation}\label{f:resolution}
  0 \rightarrow \ind_{K_\Delta}^G(V_1) \rightarrow \bigoplus_{F \subseteq \Delta, \dim(F) = d-1} \ind_{K_F}^G(V_1) \rightarrow \ldots \rightarrow \bigoplus_{F \subseteq \Delta,\dim(F) = 0} \ind_{K_F}^G(V_1) \rightarrow V_1 \rightarrow 0 \ .
\end{equation}
For a second object $V_2$ in $\Mod(G)$ we apply the left exact covariant functor $\Hom_G(-,V_2) : \Mod(G)^{\op} \rightarrow \mathrm{Vec}_k$ to this resolution and use the corresponding hypercohomology spectral sequences (cf.\ \cite{Mil} App.\ C) together with Frobenius reciprocity to obtain the following spectral sequence (replacing the earlier Mayer-Vietoris sequence).

\begin{theorem}\label{thm:spectral-seq}
   For any two objects $V_1$ and $V_2$ we have a convergent spectral sequence
\begin{equation*}
   E_1^{r,s} \Longrightarrow \Ext_G^{r+s}(V_1,V_2) \ ,
\end{equation*}
where
\begin{equation*}
   E_1^{r,s} \cong \bigoplus_{F \subseteq \Delta , \dim(F) = r} \Ext_{K_F}^s (V_1,V_2) \ .
\end{equation*}
\end{theorem}

To draw a vanishing result from the above spectral sequence we first introduce some notation from \cite{BGP}. By $\mathcal{FP}_\infty(G)$ we denote the class of objects $F$ in $\Mod(G)$ for which the functors $\Ext_\mathcal{C}^j(F,-)$, for any $j \geq 0$, preserve direct limits. Moreover, $\mathcal{FP}_\infty$-inj$(G)$ denotes the class of objects $A$ in $\Mod(G)$ such that $\Ext_G^1(F,A) = 0$ for any $F$ in $\mathcal{FP}_\infty(G)$. 

The first four Lemmas with their proofs in section \ref{sec:loc-coh} are of a general nature. We therefore have the following facts.

\begin{lemma}\label{lem:4-general}
\begin{itemize}
  \item[i.] The category $\Mod(G)$ is locally finitely presented.
  \item[ii.] The restriction of any $V$ in $\mathcal{FP}_\infty$-inj$(G)$ to a compact open subgroup $H \subseteq G$ is an injective object in $\Mod(H)$.
\end{itemize}
\end{lemma}

\begin{theorem}\label{thm:Ext-d}
   For any $V_1$ in $\Mod(G)$ and $V_2$ in $\mathcal{FP}_\infty$-inj$(G)$ we have
\begin{equation*}
  \Ext_G^j(V_1,V_2) = 0  \qquad \text{for any $j > d$},
\end{equation*}
i.e., the objects in $\mathcal{FP}_\infty$-inj$(G)$ have injective dimension $\leq d$.
\end{theorem}
\begin{proof}
It follows from Lemma \ref{lem:4-general} that $E_1^{r,s} = 0$ whenever $r > d$ or $s > 0$.
\end{proof}

\begin{remark}\label{rem:infty-coherent}
  The functor $\Ext^j_G(\ind_U^G(k),-) = \Ext^j_Uk,-) = H^j(U,-)$ on $\Mod(G)$ respects direct limits for any $j \geq 0$. Therefore Lemma \ref{lem:generators} implies that the category $\Mod(G)$ is locally type $\mathcal{FP}_\infty$ in the language of \cite{BGP} Def.\ 3.2. By the conventions of that paper in Def.\ 5.1 such categories are also called $\infty$-coherent. We emphasize that $\mathcal{FP}_\infty(G)$ is a thick subcategory of $\Mod(G)$ by \cite{Gil} Prop.\ 3.3.
\end{remark}

\end{document}